\documentclass[12pt,leqno]{amsart}
\usepackage{tikz}
\usetikzlibrary{automata,positioning}
\usepackage{amssymb,amsthm,amsmath,latexsym}
\usepackage{graphicx} 
\graphicspath{{figures/}} 
\usepackage{booktabs} 
\usepackage[font=small,labelfont=bf]{caption} 
\usepackage{amsfonts, amsmath, amsthm, amssymb} 
\usepackage{wrapfig} 

\newtheorem{theorem}{\sc Theorem}[section]

\newtheorem{lemma}[theorem]{\sc Lemma}
\newtheorem{proposition}[theorem]{\sc Proposition}
\newtheorem{corollary}[theorem]{\sc Corollary}

 \newtheorem*{thmA}{Theorem A}
 \newtheorem*{thmB}{Theorem B}
 \newtheorem*{thmC}{Theorem C}

\title{Self-similar abelian groups and their centralizers}

\author{Alex C. Dantas}
\address{Departamento de Matem\'atica, Universidade de Bras\'ilia,
Brasilia-DF, 70910-900 Brazil}
\email{(Dantas) alexcdan@gmail.com}

\author{Tulio M. G. Santos}
\address{Instituto Federal Goiano,
Campos Belos-GO, 73840-000 Brazil}
\email{(Santos) tulio.gentil@ifgoiano.edu.br}

\author{Said N. Sidki}
\address{Departamento de Matem\'atica, Universidade de Bras\'ilia,
Brasilia-DF, 70910-900 Brazil}
\email{(Sidki) ssidki@gmail.com}


\begin{document}
\maketitle

\begin{abstract}

We extend results on transitive self-similar abelian subgroups of the group of automorphisms $\mathcal{A}_m$ of an $m$-ary tree $\mathcal{T}_m$ in \cite{BS}, to the general case where the permutation group induced on the first level of the tree has $s\geq 1$ orbits. We prove that such a group $A$ embeds in a self-similar abelian group $A^*$  which is also a maximal abelian subgroup of $\mathcal{A}_m$. The construction of $A^*$ is based on the definition of a free monoid $\Delta$ of rank $s$ of partial diagonal monomorphisms of $\mathcal{A}_m$,  which is used to determine the structure of $C_{\mathcal{A}_m}(A)$, the centralizer of $A$ in $\mathcal{A}_m$. Indeed, we prove 
       $A^*=C_{\mathcal{A}_m} (\Delta(A))
           =\overline{ \Delta({B(A)})}$,
where $B(A)$ denotes the product of the projections of $A$ in its action on the different $s$ orbits of maximal subtrees of $\mathcal{T}_m$ and bar denotes the topological closure. When $A$ is a torsion self-similar abelian group, it is shown that it is necessarily of  finite exponent. Moreover, we extend recent constructions of self-similar free abelian groups of infinite enumerable rank to examples of such groups which are also $\Delta$-invariant for $s=2$. Finally, we focus on self-similar cyclic groups of automorphisms of $\mathcal{T}_m$ and compute their centralizers when $m=4.$ 
\end{abstract}

\section{Introduction}
Our purpose in this paper is to study intransitive self-similar abelian groups and their centralizers, generalizing known results for the transitive case.

Nekrashevych and Sidki characterized in \cite{NS} self-similar free abelian subgroups of finite rank of the group of automorphisms of the binary tree $\mathcal{T}_2$. Later Brunner and Sidki 
conducted in \cite{BS} a thorough study of transitive self-similar abelian subgroups of $\mathcal{A}_m$, the group of automorphisms of the $m$-ary tree $\mathcal{T}_m$, based on two facts. The first is that $\mathcal{A}_m$ is a topological group and for an abelian subgroup $G$ of $\mathcal{A}_m$, its topological closure $\overline{G}$ is again abelian. The second, when $G$ is a transitive self-similar abelian group, its closure  under the monoid $\langle x \rangle$ generated by the \textit{diagonal monomorphism} 
$$x:\alpha \mapsto (\alpha, \alpha,\dots, \alpha)$$
from $\mathcal{A}_m$ to its first level stabilizer, is again self-similar abelian.\footnote{It was asserted erroneously in Proposition 1, page 459, of \cite{BS} that the topological and diagonal closure operations commute and then that the group $A^*$ was a ``combination'' of these two closures applied to $A$. The precise statement is that in that paper $A^*$   should have been, first the diagonal closure applied to $A$ then followed by the topological closure, which in reality was the order used.}

Let $G$  be a subgroup of $\mathcal{A}_m$. Consider $P=P(G) \leqslant Sym(m)$ the permutation group, or group of activities, induced by $G$ on $Y=\{1,\dots,m\}$ and let
$$O_{(1)}=\{1,\dots,m_1\},O_{(2)}=\{m_1+1,\dots,m_1+m_2\}, \dots,$$
$$O_{(s)}=\{m_1+\cdots +m_{s-1}+1,\dots,m\}$$ of respective size $m_1,m_2,\dots, m_s,$ be the orbits of $P$ in this action; we will view these orbits as ordered in the manner as they are written here. For this set of orbits, an element $\xi$ of $Sym(m)$ is called \textit{rigid}, provided it permutes the set of orbits and is order preserving, in the sense that for all $1\leq i \leq s$ and all $k\leq l$ in $O_{(i)}$, $(k)\xi\leq (l)\xi$ in $O_{((i)\xi)}$.

The group $P$ induces transitive permutation groups $P_{(i)}$ on $O_{(i)}$, $i=1,\dots,s$, and $P$ is naturally identified with a sub-direct product of the $P_{(i)}$'s. We call $(m_1,\dots,m_s)$ the \textit{orbit-type} of $G$ and $\left(P_{(1)},\dots,P_{(s)}\right)$ the \textit{permutation-type} of $G$.
An element $\sigma$ of  $P$ decomposes as $\sigma=\sigma_{(1)} \sigma_{(2)}   \cdots  \sigma_{(s)}$, where $\sigma_{(i)} \in P_{(i)}$.

Every  $\alpha=(\alpha_1,\dots,\alpha_{m})\sigma \in G$ can be written as 
$$\alpha=(\alpha_{(1)},\dots, \alpha_{(s)})\sigma,\,\,\text{where} $$
$$\alpha_{(1)}=(\alpha_1,\dots,\alpha_{m_1}), \cdots, \alpha_{(s)}=(\alpha_{m_1+\cdots+m_{s-1}+1},\dots,\alpha_{m}).$$
Therefore, $\alpha$ can be factored in the form 
$$ \alpha=\alpha_{[1]} \cdots \alpha_{[s]}, \,\, \text{where}$$
$$\alpha_{[1]}=(\alpha_{(1)},e,\dots,e)\sigma_{(1)}, \cdots, \alpha_{[s]}=(e,\dots, e,\alpha_{(s)})\sigma_{(s)};$$
the factors $\alpha_{[i]}$ commute among themselves. 

Define the subgroups of $\mathcal{A}_m$
$$G_{[1]}=\{ (\alpha_{(1)},e,\dots ,e)\sigma_{(1)} \,\,| \,\,\alpha \in G\},$$ $$\vdots$$ $$G_{[s]}= \{ (e,\dots, e ,\alpha_{(s)})\sigma_{(s)}\,\, |\,\, \alpha \in G\}$$
and denote the group generated by the $G_{[i]}$'s by $B(G)$. Clearly, $B(G)= G_{[1]}\cdots G_{[s]}$ is a direct product of its factors and $G$ is a sub-direct product of $B(G)$.

The symmetric group $Sym(m)$ is naturally embedded in $\mathcal{A}_m$ as a group of rigid permutations of the set of maximal subtrees of $\mathcal{T}_m$; therefore, any subgroup of $Sym(m)$ in this embedding is self-similar in $\mathcal{A}_m$. 

In the intransitive setting, we need to substitute the monoide $\langle x \rangle$ by the  monoid $\Delta=\langle x_1,\dots,x_s \rangle$, where for $1\leq i \leq s,$
$x_i$ is the \textit{partial diagonal monomorphism} from $\mathcal{A}_m$ to $Stab_{\mathcal{A}_m}(1)$
$$\alpha \mapsto (e,\dots, e, \alpha,\dots, \alpha,e,\dots,e)$$ 
with $\alpha$ occurring in coordinates from the orbit $O_{(i)}$ and the trivial automorphism $e$ occurring in the other positions.
We will denote the closure of $G$ under $\Delta$ by 
$$\Delta(G)= \langle G^{\omega} \,\,|\,\, \omega \in \Delta \rangle.$$

When $A$ is a transitive self-similar abelian group, it was shown by Brunner and Sidki in \cite{BS} that the centralizer $C_{\mathcal{A}_m}(A)$   is again abelian and self-similar. Both these properties of the centralizer fail in general. The proper extension is as follows.  
\begin{thmA}
	Let  $A\leqslant \mathcal{A}_m$ be a self-similar abelian group and define $A^*=C_{\mathcal{A}_m}(\Delta (A))$. Then
	\begin{enumerate}
		\item [(i)] $\Delta(A), \Delta(B(A))$ and $\overline{\Delta(B(A))}$ are again self-similar abelian groups and of the same permutation-type as $A$;
		\item [(ii)] $A^*$ is a maximal abelian subgroup of $\mathcal{A}_m$;
		\item [(iii)] $ A^*=\overline{\Delta(B(A))}.$
	\end{enumerate}
\end{thmA}

\begin{thmB}
	There exists a finitely generated subgroup $H$ of $B(A)$ such that $A^*=\overline{\Delta({H})}$.
\end{thmB}

\begin{corollary}\label{c1}
	An abelian torsion group of infinite exponent cannot have a faithful representation as a  self-similar group.
\end{corollary}

Let $j$ be a positive integer, $\alpha=(e,\dots,e,\alpha^{x^{j-1}})(1\, \cdots \,m)$ and $D_m(j)$ the group generated by $Q(\alpha)=\{\alpha, \alpha^{x},\dots, \alpha^{x^{j-1}}\},$ the states of $\alpha.$ 
Then $D_m(j)$ is a free abelian group of rank $j$ which is transitive self-similar and diagonally closed. In contrast, in the general case, we have the following result.

\begin{thmC}
	Let $A$ be a free abelian group.
	\begin{enumerate}
		\item  [(i)] Suppose $A$ has finite rank.
		\begin{enumerate}
			\item [(a)] If $A$ is a non-transitive self-similar group then $\Delta({A})$ is not finitely generated;
			\item[(b)] Let $A$  be a transitive self-similar subgroup of $\mathcal{A}_m$. Then the self-similar representation of $A$ extends to a non-transitive self-similar representation into $\mathcal{A}_{m+1}$ having orbit-type $(m,1)$ such that $\Delta(A)$, with respect to the second representation, contains a self-similar free abelian group of infinite enumerable rank.
		\end{enumerate}		
		\item[(ii)] Suppose $A$ has infinite countable rank. Then $A$ can be realized as a self-similar group of orbit-type $(m,1),$ invariant with respect to $\Delta=\langle x_1,x_2 \rangle$.
	\end{enumerate}
\end{thmC}

To illustrate the 
second item of the theorem, let $A$ be the subgroup of $\mathcal{A}_{m+1}$ generated by
$$\alpha_1=(e, \dots, e,\alpha_{1},e)(1\,2\cdots m)(m+1),$$ $$  \alpha_{2i-1} = {\alpha_i}^{x_1} \left(i\geq 2\right)\,,\,\alpha_{2i}={\alpha_i}^{x_2} \, (i\geq 1).$$ 
Then $A$ is a self-similar free abelian group of enumerable infinite rank invariant under 
$\Delta=\langle x_1,x_2\rangle$ and has the following automata diagram.

 \begin{center}
	\begin{tikzpicture}[shorten >=2pt,node distance=2.5cm,on grid,auto]

		\node[state] (e) [] {$e$};
		\node[state] (a1) [right of=e] {$\alpha_1$};
		\node[state] (a2) [right of=a1] {$\alpha_2$};
		\node[state] (a3) [above right of=a2] {$\alpha_3$};
		\node[state] (a4) [below right of=a2] {$\alpha_4$};
		\node (i1) [above right of=a3] {$\cdots$};
		\node  (i2) [ right of=a4] {$\cdots$};
		\node (i3) [right of=a3] {$\cdots$};
		\node  (i4) [below right of=a4] {$\cdots$};

		\path [->]  (e) edge  [loop left]    node {$ $} (e)
		(a1) edge  [loop below]    node {$ $} (a1)
		(a1) edge              node {$ $} (e)
		(a2) edge              node {$ $} (a1)
		(a4) edge              node {$ $} (a2)
		(a3) edge              node {$ $} (a2)
		(i1) edge              node {$ $} (a3)
		(i3) edge              node {$ $} (a3)
		(i2) edge              node {$ $} (a4)
		(i4) edge              node {$ $} (a4)
		(a2) edge       [bend right=40]          node {} (e)
		(a3) edge       [bend right=40]          node {} (e)
		(a1) edge       [bend left=40]          node {} (e)
		(a4) edge       [bend left=40]          node {} (e);

	\end{tikzpicture}
	
	Diagram of free abelian $A$
\end{center}

We note that in the transitive case, the first construction of a free abelian self-similar group of infinite enumerable rank was given in \cite{BarSid}.

In the final section, we study self-similar cyclic groups of $\mathcal{A}_m$ and the variety of possible structures of their centralizers when $m=4$.

\section{Preliminaries}

\subsection{Self-similar groups}
Let $Y=\{1,\dots,m\}$ be a finite alphabet with $m\geq 2$ letters. The monoid of finite words $\hat{Y}$ over $Y$ has a structure
of a rooted $m$-ary tree, denoted by $\mathcal{T}(Y)$ or $\mathcal{T}_m$. The incidence relation on $\mathcal{T}_m$ is given by: $(u,v)$ is an edge if and only if there exists a letter $y$ such that $v=uy$. The empty
word $\emptyset$ is the root of the tree and the level $i$ is the set of all  words of length $i$. 

\noindent The automorphism group $\mathcal{A}_{m}$ (or ${Aut}\left( \mathcal{T}_m\right) $)
of $\mathcal{T}_{m}$ is isomorphic to the restricted wreath product
recursively defined as $\mathcal{A}_{m}=\mathcal{A}_{m}\wr Sym(Y)$. An automorphism $\alpha $ of $%
\mathcal{T}_{m}$ has the form $\alpha =(\alpha _{1},\dots,\alpha _{m})\sigma
(\alpha )$, where the state $\alpha _{i}$ belongs to $\mathcal{A}_{m}$ and $\sigma :\mathcal{A}_{m}\rightarrow Sym(Y)$ is the permutational
representation of $\mathcal{A}_{m}$ on $Y$, seen as first level of the tree $\mathcal{T}_{m}$.
The action of $\alpha =(\alpha _{1},\dots,\alpha _{m})\sigma
(\alpha ) \in \mathcal{A}_m$ on a word $y_{1}y_{2}\cdots y_{n}$ of length $n$ is given recursivelly by  $\left( y_{1}\right) ^{\sigma (\alpha
	)}\left( y_{2}\cdots y_{n}\right) ^{\alpha _{y_{1}}}$.

\noindent The recursively defined subset of $\mathcal{A}_m$
\begin{equation*}
	Q(\alpha )=\{\alpha\}\cup
	_{i=1}^{m}Q(\alpha _{i})
\end{equation*}%
is called the set of \textit{states} of $\alpha $. A subgroup 
$G$ of $\mathcal{A}_{m}$ is \textit{self-similar} (or \textit{state-closed}, or \textit{functionally recursive}) if $Q(\alpha )$ is a subset of $G$ for all $\alpha $ in $G$ and is \textit{transitive} if its action on $Y$ is transitive. 

\subsection{Self-similarity data}
A \textit{virtual endomorphism} of an abstract group $G$ is a homomorphism $f:H \to G$ from a subgroup $H$ of finite index in $G$. Let $G$ be a group and consider 
\begin{equation*}
	\mathbf{H=} \left( H_{i} \leqslant G \mid \left[ G:H_{i}\right] =m_{i}\text{ }%
	\left( \text{ }1\leq i\leq s\right) \right) \text{,}
\end{equation*}%
\begin{equation*}
	\mathbf{m=}\left( m_{1},\dots,m_{s}\right) ,\text{ }m=m_{1}+\dots+m_{s},
\end{equation*}
\begin{equation*}
	\mathbf{F}=\left( f_{i}:H_{i}\rightarrow G \,\, \text{virtual endomorphisms}\mid 1\leq i\leq s\right) \text{;
	}
\end{equation*}
we will call  $\left( \mathbf{m},\mathbf{H,F}\right) $ a $G$-\textit{data} or \textit{data for} $G$. The $\mathbf{F}$-\textit{core} is the largest subgroup of $ \cap _{i=1}^{s}H_{i}$ which is normal in $G$ and $f_i$ - invariant for all $i=1,\dots,s$


The following approach to produce intransitive self-similar groups was given in \cite{DSS}.

\begin{proposition} Given a group $G$, $m\geq 2$ and a $G$-data $\left( 
	\mathbf{m},\mathbf{H,F}\right) $. Then the data provides a self-similar
	representation of $G$ on the $m$-tree with kernel the $\mathbf{F}$-core, 
	\begin{equation*}
		\langle K\leqslant \cap _{i=1}^{s}H_{i}\mid K\vartriangleleft
		G,K^{f_{i}}\leqslant K,\forall i=1,\dots,s\rangle \text{.}
	\end{equation*}
\end{proposition}

A $G$-data $(\bf{m}, \bf{H}, \bf{F})$ is said to be \emph{recurrent} if each virtual endomorphism $f_{i}$ is an epimorphism and the $\bf{F}$-core is trivial.
We say that a recurrent $G$-data $(\bf{m}, \bf{H}, \bf{F})$ is \emph{strongly recurrent} if
$$f_{i}: H_{i} \cap \bigcap_{j \neq i} ker(f_{j}) \rightarrow G$$
is an epimorphism for any $f_{i}$ in ${\bf F} = (f_{1}, \dots, f_{s})$.

There are recurrent $G$-data $(\bf{m}, \bf{H}, \bf{F})$ which are not strongly recurrent. Indeed, consider $A$ the group generated by the double-adding machine $a=(e,a,e,a)(1\,2)(3\,4)$\footnote{This automorphism was first considered in Section 2.1 of \cite{BerSid}. Its generalization for the $m$-tree is given here in 
	Subsection 5.1.1.}, then $Fix_A(1)=Fix_A(3)=\langle a^2\rangle$ and the $A$-data $\left((2,2),(Fix_A(1),Fix_A(3)),(\pi_1,\pi_3)\right)$ is recurrent but it is not strongly recurrent.

\begin{center}
	\begin{tikzpicture}[shorten >=0.8pt,node distance=2.2cm,on grid,auto] 
		
		\node[state] (a)  {$a$};
		\node[state] (e) [right of=a] {$e$};
		
		\path [->]
		(a) edge              node {$1|2,3|4$} (e)
		(a) edge    [loop left]        node {$2|1,4|3$} (b)
		(e) edge      [loop right]   node {$1|1,2|2,3|3,4|4$} (e);
		
	\end{tikzpicture}
\end{center}

\begin{center}
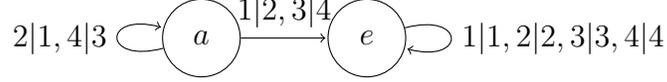
	Diagram of the double adding machine $a$
\end{center}

\section{Operators acting on automorphisms of ${m}$-trees}

\noindent 

The fact that partial monomorphisms play an important role here, we approach the subject more generally. Define
$$ Func(m)=\{f:\mathcal{A}_m \to \mathcal{A}_m \text{   function}\},$$ $$End(m)=\{f:\mathcal{A}_m \to \mathcal{A}_m \text{   endomorphism}\},$$
$$ Mon(m)=\{f:\mathcal{A}_m \to \mathcal{A}_m \text{   monomorphism}\}.$$

As usual, $Func(m)$ is closed under two binary operations: the composition of functions ``$\mathbf{\cdot}$" defined by
$$(a)(f\mathbf{\cdot}g)=((a)f)g$$
and the non-commutative `` $\mathbf{+}$ ''  defined by 
$$(a)(f+g)=((a)f)((a)g),$$
for all $f,g \in Fun(m)$ and $a\in \mathcal{A}_m.$

$End(m)$ is closed under ``$\mathbf{\cdot}$" and $Mon(m)$ is a submonoid of $End(m)$. We note that for $f,g,h \in Func(m),$
$$h\cdot(f+g)=h\cdot f+h\cdot g, $$
and
$$(f+g)\cdot h =f\cdot h +g\cdot h$$
provided   $h \in End(m).$

\begin{lemma}\label{end}
	Let $f,g \in End(m)$ and let $h=f+g$.
	\begin{enumerate}
		\item [(i)] Then, $h\in End(m)$ iff $({\mathcal{A}_m})^f$ commutes element-wise with $({\mathcal{A}_m})^g.$
		\item[(ii)] Suppose $h \in End(m)$. Then $h \in Mon(m)$ iff 
		${\mathcal{A}_m}^f \cap {\mathcal{A}_m}^g =\{e\}.$ 
	\end{enumerate}
\end{lemma}

The proofs are straightforward.

\subsection{Monomorphisms of $\mathcal{A}_m$.}
$Mon(m)$ contains the group $(\mathcal{A}_m)\kappa$ induced by conjugation of $\mathcal{A}_m$ on itself. Since $m\geq 2$, the center of $\mathcal{A}_m$ is trivial, and so, $(\mathcal{A}_m)\kappa$ is isomorphic to $\mathcal{A}_m$. The tree structure induces certain additional momomorphisms of $\mathcal{A}_m$ defined as follows.  First, a \textit{connecting set} 
$M$ of $\mathcal{T}(Y)$ is a subset of vertices of $\mathcal{T}(Y)$ such that every element of $\mathcal{T}(Y)$ is comparable to some element of $M$ and 
different elements of $M$ are incomparable.
Some examples for $Y=\{1, 2\}$ are: $M =\{\emptyset\}$; $M=Y$; $M=\{1, 21, 221,\dots , 2^i\cdot 1,\dots \}$.

Given a connecting set $M$ and $N$ a subset of $M$ (called \textit{partial connecting set}), define the monomorphism
$$\mathbf{\delta}_N :\mathcal{A}_m \to \mathcal{A}_m,$$   
$$\, \, \, \, \, \, \, \, \, \, \, \, \, \, \, \, \, \, \, \, \, \, \, \, \, \, \, \, \, \, \, \, a \mapsto (a)\mathbf{\delta}_N=b $$
where $b$ is such that $b_u =a$ for all $u \in N$ and $e$ for $u$ in $M \setminus N$. 
Then $\delta_N$ is in $Mon(\mathcal{A}_m).$  When $w \in \hat{Y}$, we simplify the notation $\delta_{\{w\}}$ to  $\delta_{w}$. Define $\delta(Y)$ to be the monoid $\langle \delta_i \,\,|\,\, i \in Y \rangle$. When $Y$ is fixed we simplify  $\delta(Y)$ to $\delta$ and denote the image of a subgroup $H$ of $\mathcal{A}_m$ under the action of $\delta$ by  $\delta(H)=H^{\delta}=\langle (a)\cdot w  \text{   } | \text{   }  a \in H\text{   and   } w \in \delta \rangle$.

\subsection{The monoid $\delta(Y)$}

\begin{lemma}
	For all $u,v \in \hat{Y},$ 
	$\delta_u\delta_v=\delta_{vu}$
	and therefore, 
	$$\delta=\langle \delta_i  \text{   } | \text{   }  i \in \hat{Y}\rangle$$
	is a free monoid of rank $m$.
	
\end{lemma}

\begin{proof}
	We check
	$$ (a)\delta_u\delta_v=((a)\delta_u)\delta_v=(a)\delta_{vu},$$
	for all $a \in \mathcal{A}_m$ and freeness follows. 
\end{proof}

\begin{lemma}
	Let $U=\{u_i\,\, |\,\,  1\leq i \leq k\}$ be a set of incomparable elements of $\mathcal{T}_m$. Then, $\sum_{u \in U} \delta_u \in Mon(m)$,
	$\{\delta_{u} \,\,|\,\, u \in U\}$ is a commutative set and
	$$(\mathcal{A}_m)^{\delta_v} \cap \langle(\mathcal{A}_m)^{\delta_u}\,\, |\,\, u \in U, u\neq v\rangle = \{e\}.$$
\end{lemma}

\begin{proof}

	The argument is as that of Lemma \ref{end}.
\end{proof}

\begin{lemma}\label{act}
	(Permutability Relations) Let $r \in \mathcal{A}_m$ and $w \in \mathcal{T}_m.$ Then 
	$$\delta_w\cdot (r)\kappa=(r_w)\kappa \cdot \delta_{(w)r} $$
	and therefore the following monoid factorization holds,
	$$\langle (\mathcal{A}_m)\kappa, \delta(Y) \rangle =(\mathcal{A}_m)\kappa \cdot \delta(Y). $$
\end{lemma}
\begin{proof}
	Let $|w|=k$. The development of $r$ at the $k$-th level of the tree is $r=(r_u \text{     }| \text{     } |u|=k)\sigma$ for some permutation $\sigma$ of the $k$-th level. Then as
	$$(a)(\delta_w)=(e,\dots,e,a,e,\dots,e) $$
	an element of $Stab_{\mathcal{A}_{m}}(k)$ and has $a$ in the $w$-th coordinate, we compute:
	
	\begin{equation*}
		\begin{split}
			(a)(\delta_w)\cdot (r)\kappa &= {((a)\delta_w)}^r\\
			&= {(e,\dots,e,a,e,\dots,e)}^r\\
			&= {(e,\dots,{a}^{r_w},e,\dots,e)}^{\sigma} \\
			&= (e,\dots,e,a^{r_w},e,\dots,e),
		\end{split}
	\end{equation*}
	where in the last equation $a^{r_w}$ is in the $(w)^{\sigma}=(w)r$ coordinate. Thus, 
	$$(a)(\delta_w)\cdot (r) \kappa =(a^{r_w})(\delta_{(w)r})=(a)((r_w)\kappa\cdot (\delta_{(w)r})),$$
	$$\delta_w\cdot (r)\kappa =(r_w)\kappa \cdot (\delta_{(w)r}). $$
\end{proof}

\subsection{Partial diagonal monomorphisms in $\delta(Y)$}

Let $\pi=\{Y_j \text{   } | \text{   } j=1,\dots,s\}$ be a partition of $Y$.
Define $\delta(\pi)=\langle\delta_N \text{   } | \text{   } N \in \pi \rangle.$

In the particular case where $\pi$ is the set of orbits $\{O_{(1)},\dots,O_{(s)}\}$ of $G \leqslant \mathcal{A}_m$ on $Y$, we denote the partial diagonal monomorphisms associated to $G$ and its set of orbits by
$$x_1:=\lambda_{O_{(1)}},\dots,x_s:=\lambda_{O_{(s)}},$$
the monoid  generated by $x_1,\dots,x_s$ by $\Delta$ and denote the group 
$\langle G\cdot w | w \in {\Delta}\rangle $ by $\Delta_{}(G)$ or by $G^{\Delta_{}}$. 

Note that $\Delta(G)$ is a subgroup of $L(P(G))$, the layer-closure of $P(G)$ in $\mathcal{A}_m$, which is the subgroup formed by all automorphisms in $\mathcal{A}_m$ whose states have activities belonging to $P(G)$. 

Let us denote the elements of the orbit $O_{(i)}$ by $i1,i2,\dots, im_i.$

\begin{lemma}
	Let $r \in \mathcal{A}_m$ and $x_i$ in $\Delta$. 
	\begin{enumerate}
		\item [(i)] Then
		$$x_i\cdot (r)\kappa=[(r_{i1})\kappa \cdot (\delta_{(i1)r})] \cdot [(r_{i2})\kappa \cdot (\delta_{(i2)r})]\cdots [(r_{im_i})\kappa \cdot (\delta_{(im_i)r})].$$
		\item[(ii)] If  the states $r_{i1},\dots,r_{im_i}$ of $r$ centralize an element $a$ of $\mathcal{A}_m$, then 
		$$(a)x_i \cdot (r)\kappa=(a)x_l,$$ where $O_{(l)}=O_{((i)r)}$.
	\end{enumerate}
\end{lemma}

\begin{proof}
	Let $a \in \mathcal{A}_m$. Then 
	
	\begin{equation*}
		\begin{split}
			(a)x_i \cdot (r)\kappa &= [(a)\delta_{i1}(a) \delta_{i2}\cdots(a)\delta_{im_i}]^{r}\\
			&= [(a)\delta_{i1}]^r[(a)\delta_{i2}]^r\cdots[(a)\delta_{im_i}]^{r}\\
			&= [(a^{r_{i1}})\delta_{(i1)r}][(a^{r_{i2}})\delta_{(i2)r}]\cdots [(a^{r_{im_i}})\delta_{(im_i)r}]. 
		\end{split}
	\end{equation*}
	In the last equality we used Lemma \ref{act}. 
	
	\noindent Item (ii) is consequence of (i). 
\end{proof}

\subsection{Some relations between the operations, $\Delta$, $Stab$,  ${\text{}}^{\overline{\,\,\,}}$ and $B$.}

Let $A$ be an abelian self-similar subgroup of $\mathcal{A}_m$.
\begin{lemma}
	Let $x_i \in \Delta$ and $k\geq 1$. Then, $Stab_{A^{x_i}}(k)= {\left( Stab_{A}(k-1) \right)}^{x_i}$.
	
\end{lemma}

\begin{proof}
	Straightforward. 
\end{proof}

\begin{lemma}\label{s2}
	
	\begin{enumerate}
		\item [(i)] $\Delta(\overline{A}) \leqslant \overline{\Delta(A)};$
		\item [(ii)] for $m=2$ and $A=\langle a=(1\,2)\rangle$, the above inequality is proper.
	\end{enumerate}
\end{lemma}

\begin{proof}

	\begin{enumerate}
		\item [(i)] The topological closure of $A$ is the infinite product
		$$\overline{A}=A\cdot Stab_{A}(1)\cdots Stab_{A}(i)\cdots$$
		Therefore, for any $w \in \Delta$,
		$${(\overline{A})}^w=
		A^w\cdot {\left(Stab_{A} (1)\right)}^w\cdots {\left(Stab_A(i)\right)}^w\cdots$$
		which, by the previous lemma, is a subset of $\overline{A^w}$.
		\item[(ii)] Since for $A=\langle(1\,2)\rangle \leqslant \mathcal{A}_2$, $\overline{A}=A$, 
		we have 
		$$\Delta (\overline{A})=\{a^{p(x)}\,|\, p(x) \in \mathbb{Z}/ 2\mathbb{Z}\,[x]\},$$
		whereas,
		$$\overline{\Delta (A)}=\{a^{f(x)}\,|\, f(x) \in \mathbb{Z}/ 2\mathbb{Z}\,[[x]]\}.$$
	\end{enumerate}

\end{proof}

\begin{lemma} \label{s1}  Let $w \in \Delta$.
	\begin{enumerate}
		\item [(i)] $B(A^w)=A^w$ if $|w|\geq 1$;
		\item[(ii)]
		$$B(\Delta(A))=B(A)\cdot\Delta(A);$$
		$$\Delta(B(\Delta(A)))
		=B(\Delta(A)).$$
	\end{enumerate}
\end{lemma}

\begin{proof}
	\begin{enumerate}
		\item [(i)] $(A^{w'x_i})_{[k]} = 1$ if $k$ is different from $i,$ and
		$(A^{w'x_i})_{[k]}=A^{w'x_i}$ if $k=i$.
		\item [(ii)] 
		\begin{equation*}
			\begin{split}
				B(\Delta(A)) &= B(A) \langle B(A^w)\,\,|\,\, |w|\geq1 \rangle\\
				&=  B(A) \langle A^w\,\,|\,\, |w|\geq 1 \rangle \\
				&= B(A) A \langle A^w\,\,|\,\, |w|\geq 1 \rangle\\
				&= B(A)\Delta(A).
			\end{split}
		\end{equation*}
	\end{enumerate}
\end{proof}

\section{The centralizer of self-similar abelian groups}

We start with a description of the structure of centralizers of self-similar abelian groups and then proceed to prove Theorem A.

The very first step is a description of the structure of the centralizer of subgroups of finite symmetric groups.

\begin{lemma}\label{c}
	Let $Q$ be a subgroup of $Sym(m)$, having the set of orbits $O_{(i)}$, $1\leq i \leq s$. Then 
	the centralizer of $Q$ in $Sym(m)$ has the structure 
	$$C_{Sym(m)} (Q) =C_{Sym(m_1)}(Q_{(1)})\cdots C_{Sym(m_s)} (Q_{(s)})\cdot S(Q),$$
	where $S(Q)$ is the subgroup of rigid permutations in $Sym(m)$ with respect to the given set of orbits.
	Let $\{J_1,\dots,J_t\}$ be the partition of the orbits of $Q$ having equal length. Then, $S(Q)$ is isomorphic to the direct product $Sym(J_1) \times \cdots \times Sym(J_t)$. In case $Q$ is abelian then $C_{Sym(m)}(Q)=B(Q)S(Q).$
\end{lemma}

Given a subgroup $G$ of $\mathcal{A}_m$, we write $C(G)$ for $C_{\mathcal{A}_m}(G)$ and define
$$S(G)=\{ \xi \in S(P(G)) \,\,|\,\, C_{}(G)\cap Stab_{\mathcal{A}_m}(1)\xi \,\,\text{ non-empty}\}.$$
Since $S(P(G)) = S(P(B(G))$,  any lifting of $\xi$ in  $S(G)$ to $C_{\mathcal{A}_m}(G)$ may be modified by an element of $B(G)$ so as to induce a rigid permutation of the orbits, in this sense that we choose each $\xi$ to be rigid 
and define $R(G)$ to be the group generated by these liftings. We note that any two liftings of $\xi$ differ by an element of $Stab_{{G}} (1)$. If $\xi$ itself is in $C_{\mathcal{A}_m}(G)$ then ${G_{[i]}}^{\xi}=G_{[j]}$ where ${O(i)}^{\xi}=O_{(j)}$, and therefore $\xi$ normalizes $B(G)$.

\begin{proposition}\label{teo1}  Let $A$ be a self-similar abelian subgroup of $\mathcal{A}_m$ and recall $B(A)$, $C(A)$, $R(A)$. Then,
	\begin{enumerate}
		\item [(i)] $Stab_{A}(1)$ is a subgroup of $A^{x_1}\cdot A^{x_2}\cdots A^{x_s}$.
		\item [(ii)] $C(A)=Stab_{C(A)} (1) B(A) R(A)$;
		\item [(iii)]  the subgroup $Stab_{C(A)} (1) B(A)$ is normal in $C(A)$;
		\item[(iv)]  $B(A)$ centralizes $Stab_{C(A)} (1)$;
		\item[(v)] $C(A)$ and $Stab_{C(A)} (1)$ are $\Delta$ invariant;
		\item [(vi)] $\Delta(A)$, ${\Delta}(B(A))$ and $\overline{\Delta(B(A))}$ are self-similar abelian groups of the same permutation-type as $A$.
	\end{enumerate}
\end{proposition}

\begin{proof}
	
	\begin{enumerate}
		\item [(i)] An element of $Stab_{A}(1)$ has the form $a=(a_{(1)}, a_{(2)},\dots, a_{(s)})$. As $P_{(i)}$ is transitive on $O_{(i)}$, conjugation of $a$ by elements $A_{[i]}$ shows that the entries of $a_{(i)}$ are all equal.
		\item [(ii)] Follows from the fact that $C(A)$ modulo $Stab_A(1)$ is isomorphic to a subgroup of $C_{Sym(m)}(P(A))$, as in Lemma \ref{c}.
		\item [(iii)]  Given $r=(r_1,\dots,r_m)\xi \in R(A)$, where $\xi \in S(P(A))$ and $h \in A_{[i]}$ written as $h=h'\cdot\sigma$, where $h'\in Stab_{\mathcal{A}_m}(1)$ and $\sigma \in P_{(i)}$, we have
		$\sigma^\xi \in P_{(i^\xi)}$. As there exists $g \in A_{[i^\xi]}$ such that $g=g'\cdot \sigma^\xi$, we have 
		$(h^r)\cdot g^{-1}=(h')^r\cdot (g')^{-1} \in Stab_{C(A)}(1)$
		and therefore the subgroup
		$Stab_{C(A)} (1) B(A)$ is normal in $C(A)$. 
		\item [(iv)]  Let $f=(f_1,f_2,\dots,f_m)=f_{[1]} f_{[2]}\cdots f_{[s]}  \in Stab_{\mathcal{A}_m} (1)$ and $a=a_{[1]}\cdots a_{[s]} \in A.$ 
		Then, 
		$$f^a= {f_{[1]}}^{a_{[1]}}\cdots {f_{[s]}}^{a_{[s]}}.$$
		Thus,
		$f \in C(A)$ iff $f_{[i]}$ centralizes $A_{[i]}$ 
		for all $i$.
		As $f_{[i]}$ centralizes $A_{[j]}$ for all $j$ different $i$,
		it follows that 
		$f \in C(A)$ iff $f_{[i]}$ centralizes $B(A)$
		for all $i$.
		
		\item [(v)] Given $c \in C(A)$, easily $c^{x_i} \in C(A).$
		Therefore, 
		${C(A)}^{\Delta}=C(A)$.
		Since 
		$(Stab_{C(A)} (1) )^{\Delta} \leqslant {C(A)}^{\Delta}=C(A)$
		and ${(Stab_C (1))}^w \leqslant Stab_{C(A)} (1) $ for $w\in \Delta$ of length at least 1, we have 
		$$(Stab_{C(A)} (1) )^{\Delta}=Stab_{C(A)} (1) ).$$
		
		\item[(vi)] 
		Since $A$ is a self-similar abelian group, $\Delta(A)$ is also self-similar abelian.
		By induction on the length of elements of $\Delta$, the fact that
		$\Delta(B(A))$ is abelian is reducible to 
		$A_{[i]}$ commutes with $(A_{[j]})^w$
		for all $i,j$. 
		Clearly,  $A_{[i]}$ commutes with $(A_{[j]})^w$ if $w=1$ or $w=w'\cdot x_k$ for $k \neq i$.
		In case $k=i$,  then
		$A_{[i]}$ commutes with $(A_{[j]})^w$
		iff $A_{[i]}$ commutes with $(A_{[j]})^{w'}$.
		Further, since $P(\Delta (B(A)))=P(\Delta(A))=P(A)$, it follows that $\Delta(B(A))$ and $\Delta(A)$ are of the same permutation-type as $A$. Since the topological closure of abelian groups are also abelian, $\overline{\Delta(B(A))}$ is abelian and is also self-similar.
	\end{enumerate}
\end{proof}

Note that items (i) and (ii) above hold for $A$ abelian without the extra condition of being self-similar.

\subsection{The case where $R(A)$ consists of rigid permutations}

Naturally, the complexity of $C(A)$ depends upon that of $R(A)$. In case 
$R(A)=S(A)$, we obtain finer structural results.

\begin{proposition}
	Suppose $R(A)=S(A).$ 
	Then,
	\begin{enumerate}
		\item [(i)] $B(A)$ is normal in $C(A)$ and $\Delta (B(A))$ is normalized by $\Delta(R(A)).$
		\item [(ii)] Define $H= \Delta(B(A)) \Delta(R(A))$. Then
		$ \overline{\Delta(R(A))}$ normalizes $\overline{\Delta({B(A)})}$
		and 
		$\overline{H}=\overline{\Delta ({B(A)})} \,\, \overline{\Delta(R(A))}.$
	\end{enumerate}
\end{proposition}

\begin{proof}
	\begin{enumerate}
		\item [(i)] As $R(A)=S(A)$, conjugation by elements of $R(A)$ induce permutations of the set of $A_{[i]}$'s and therefore $R(A)$ normalizes $B(A)$. Hence $B(A)$ is a normal subgroup of $C(A)$.
		We calculate for $b \in B(A)$ the conjugates of $b^v$ by $\xi^w$ for $b=a_{[i]} \in A_{[i]}$ and $v,w \in \Delta.$ 
		As $B(A)$ is normal in $C(A)$, we may assume $|v|>0$ and write $v=v'x_i$ for some $i$. Also, since for any $g \in \mathcal{A}_m$, 
		${\left(g^{x_k}\right)}^{\xi}=g^{x_l}$  where $l=k^{\xi}$,
		we may assume $|w| >0$ and write $w=w'x_j$. 
		Now check that ${\left(b^{w'x_i}\right)}^{{\xi}^{v'x_j}}
		=\left(b^{w'}\right)^{x_i}$, if $j=i$ 
		and $ {\left(b^{w'x_i}\right)}^{{\xi}^{v'x_j}} = b^{w'x_i}$, otherwise. 
		The result follows by induction on the lengths of $v$ and $w$.
		\item[(ii)] Since $P(A)$ and $P(S(A))$ intersect trivially it follows that their layered closures $L(P(A))$,
		$L(P(S(A)))$ also intersect trivially. Therefore, for all $k\geq 0$,
		the $k$-th stabilizers of $\Delta(B(A))$ and $\Delta(R(A))$ intersect trivially and thus,
		$$Stab_{H}(k) =
		Stab_{\Delta(B(A))} (k)  \cdot Stab_{\Delta(R(A))} (k) .$$
		Hence,
		$\overline{\Delta(R(A))}$ normalizes
		$\overline{\Delta(B(A))}$. 
	\end{enumerate}
\end{proof}

We prove Theorem A in a more detailed form. 

\begin{theorem}\label{cent}
	Let  $A\leqslant \mathcal{A}_m$ be an abelian self-similar group and let $A^*=C_{\mathcal{A}_m}(\Delta(A))$. Then
	\begin{enumerate}
		
		\item[(i)]  $A^*$ leaves each orbit $O_{(i)}$ invariant (that is, $R(\Delta(A))=1$) and  $$A^*=Stab_{A^*}(1)B(A);$$
		
		\item[(ii)] $Stab_{A^*}(1)={(A^*)}^{x_1}{(A^*)}^{x_2}\cdots{(A^*)}^{x_s};$ 
		\item[(iii)]
		$A^*=\overline{\Delta({B(A)})};$
		\item [(iv)] $A^*$ is a maximal abelian subgroup of $\mathcal{A}_m$ and is the unique one containing $\Delta(A)$.
	\end{enumerate}
\end{theorem}

\begin{proof} 
	Denote $\Delta(A)$ by $K$.
	\begin{enumerate}
		\item[(i)] We need to prove $R(K)=1$. Suppose, by contradiction, that  $R(K) \neq 1$ and let $r=(r_{(1)},\dots,r_{(s)})\xi \in R(K)$ with $\xi \neq 1$. Let $\alpha \in A^*$. Since $\xi\neq 1$ there are $O_{(i)}\neq O_{(j)}$ such that $\xi: O_{(i)} \mapsto O_{(j)}$. Then ${(\alpha^{x_i})}^r=\beta^{x^j}$, where $\beta=(\beta_{(1)},\dots,\beta_{(s)})$ and
		$\beta_{(k)}=1$ for $k \neq j$, 
		$\beta_{(j)}= (\alpha,\dots,\alpha)^{r_{(i)}}$, a contradiction. 
		Therefore 
		$$A^*=Stab_{A^*}(1)B(K).$$
		By Lemma \ref{s1}, $B(K)=B(A)\Delta(A)$ and therefore, 
		\begin{equation}\label{e1}
			A^*=Stab_{A^*}(1)B(A).
		\end{equation}
		
		By Proposition \ref{teo1} (iv), $B(A)$ is central in $A^*$.
		\item[(ii)] 
		Let  $\gamma=(\gamma_{(1)},\dots,\gamma_{(s)}) \in Stab_{A^*}(1)$, $w \in \Delta$ and $a \in A$. Then,
		$$[\gamma, a^{wx_i}]=
		(e,\dots,e,[\gamma_{(i)},a^w],e,\dots,e),  \,\,
		\gamma_{(i)} \in C(A^{\Delta})=A^*.$$
		Therefore,
		$\gamma \in {(A^*)}^{x_1}\cdots{(A^*)}^{x_s}.$ Thus, 
		\begin{equation}\label{e2}
			Stab_{A^*}(1)={A^*}^{x_1}{A^*}^{x_2}\cdots{A^*}^{x_s}.
		\end{equation}
		
		\item[(iii)] We substitute \eqref{e1} in \eqref{e2} and collect $B(A)^{x_i}$ to the right to obtain

		$$Stab_{A^*} (1) $$
		$$=\prod_{w \in \Delta,\,\,|w|=1}\left( Stab_{A^*}(1)\right)^w \prod_{w \in \Delta,\,\,|w|=1}\left( B(A)\right)^w.$$
		Thus,
		$$Stab_{A^*} (1)=Stab_{A^*} (2) \prod_{w \in \Delta,\,\,|w|=1}\left( B(A)\right)^w.$$
		More generally, 
		
		\begin{equation}\label{e3}
			Stab_{A^*} (1)= \prod_{w \in \Delta,\,\,|w|=k}\left( Stab_{A^*}(1)\right)^w  \prod_{w \in \Delta,\,\,|w|=k}\left( B(A)\right)^w.
		\end{equation}
		and 
		
		\begin{equation}\label{e4}
			Stab_{A^*} (1)= Stab_{A^*} (k+1)\prod_{w \in \Delta,\,\,|w|=k}\left( B(A)\right)^w.
		\end{equation}

		\noindent Let $g \in A^*$. Then, as $A^*=Stab_{A^*}(1) B(A)$, there exists  $b^{(0)} \in B(A)$ such that 
		$$g\left(b^{(0)}\right)^{-1} \in Stab_{A^*} (1)$$
		and using \eqref{e3},
		there exists  $b^{(1)} \in \prod_{w \in \Delta, |w|=1} {(B(A))}^w$ such that
		
		$$g{\left(b^{(0)}\right)}^{-1} {\left(b^{(1)}\right)}^{-1}  \in \prod_{w \in \Delta,\,\, |w|=1} \left(Stab_{A^*}(1)\right)^w \leqslant Stab_{A^*}(2).$$
		Iterating, we produce $b^{(k)}$ in $\prod_{w \in \Delta,\,\, |w|=k}{(B(A))}^w$
		such that 
		$$g{\left( b^{(0)}\right)}^{-1}{\left( b^{(1)}\right)}^{-1}\cdots {\left( b^{(k)}\right)}^{-1} \in Stab_{A^*} (k+1).$$
		Thus we produce the infinite product 
		$$b= \cdots  b^{(k)} \cdots b^{(1)} b^{(0)} \in \overline{\Delta({B(A)})}$$ 
		and
		$$g{b}^{-1} \in \bigcap_{k \geq 0} Stab_{A^*} (k)=1$$ 
		Hence, $g =b$ and $A^*=\overline{\Delta({B(A)})}.$
		
		\item[(iv)] Suppose $M$ is a maximal abelian subgroup of $\mathcal{A}_m$ and $\Delta(A)\leqslant M.$ Then, $C_{\mathcal{A}_m}(M)\leqslant C_{\mathcal{A}_m}(\Delta(A))=A^*.$
		Therefore, $C_{\mathcal{A}_m}(M)$ is abelian and so, $C_{\mathcal{A}_m}(M)=M \leqslant A^*$ and $M=A^*.$
	\end{enumerate}
\end{proof}

Next we prove 

\begin{thmB}
	There exists a finitely generated subgroup $H$ of $B(A)$ such that $A^*=\overline{\Delta(H)}$.
\end{thmB} 

\begin{proof}
	Choose a generating set  
	$\{\sigma_{ij} \,\,|\,\, 1\leq j \leq l_i\}$ for each $P_{(i)}$ and choose $\beta_{ij} \in  B(K)$  having activity $\sigma_{ij}$. 
	Define $H=\langle \beta_{ij}\,\, |\,\,1\leq i\leq s,\,\, 1\leq j \leq l_i \rangle.$
	Then 
	$$B(A)\leqslant Stab_{B(A)} (1) H$$
	
	and by Proposition \ref{teo1} (i),
	
	$$Stab_{B(A)} (1) \leqslant {B(A)}^{x_1}\cdots {B(A)}^{x_s}.$$
	
	Therefore, 
	$$A^* = Stab_{A^*} (1) B(A)
	= Stab_{A^*} (1) H,$$
	
	$$ Stab_{A^*} (1) ={(A^*)}^{x_1}\cdots{(A^*)}^{x_s}$$
	$$= {(Stab_{A^*} (1) H)}^{x_1} \cdots {(Stab_{A^*} (1) H)}^{x_s}$$
	$$=\prod_{w \in \Delta,\,\, |w|=1}
	\left( Stab_{A^*} (1)\right) ^w 
	\prod_{w \in \Delta,\,\, |w|=1} H^w.$$
	Thus, we may substitute $H$ for $B(A)$ in our arguments in the previous theorem to  obtain
	$A^*= \overline{\Delta(H)}.$
	
\end{proof}

An equivalent form of Corollary \ref{c1} is

\begin{corollary}
	Let $A$ be a torsion self-similar abelian group. Then $A$ has finite exponent. 
\end{corollary}

\begin{proof}
	As $A^*= \overline{\Delta(H)}$,
	and $H$ has finite order, it follows that 
	$A^*$ has the same exponent as $H$ and therefore so does $A$.
\end{proof}

\subsection{Free Abelian Groups}

\begin{proposition}\label{sr}
	Let $A$ be a self-similar abelian group of degree $m = m_{1} + \cdots + m_{s}$ and orbital-type ${\bf m} = (m_{1},\dots, m_{s})$, with $s \geq 2$. Then $A = \Delta(A)$ if and only if there exists a strongly recurrent $A$-data $({\bf m}, {\bf H}, {\bf F})$. 
\end{proposition}

\begin{proof}
	If $A = \Delta(A)$, it is enough to consider the $A$-data $({\bf m}, {\bf H}, {\bf F})$ defined by
	$$H_{1} = Fix_{A}(1), H_{2} = Fix_{A}(m_1+1), \dots, H_{s} = Fix_{A}(m_{1} + \cdots + m_{s - 1}+1),$$
	and
	$$f_{1} = \pi_{1}, f_{2} = \pi_{ m_{1}+1}, \dots, f_{s} = \pi_{m_{1} + \cdots + m_{s - 1}+1},$$
	where $\pi_i$ is the projection on the $i$-th coordinate.  
	
	Now, suppose that we are given a strongly recurrent data $({\bf m}, {\bf H}, {\bf F})$ for $A$. Let $T_{1}, \dots, T_{s}$ transversal of $H_{1}, \dots, H_{s}$ in $A$, respectively. We will show that given $a \in A$ there exists $b \in A$ such that $b^{\varphi} = a^{\varphi x_{i}}$ for all $i = 1, \dots, s$, where $\varphi: A \rightarrow \mathcal{A}_{m}$ is the representation induced by the $A$-data $({\bf m}, {\bf H}, {\bf F})$ and the transversal $T_{1}, \dots, T_{s}$. Indeed, since $f_{i}: H_{i} \cap \bigcap_{j \neq i} ker(f_{j}) \rightarrow A$ is onto there exists $b \in H_{i} \cap \bigcap_{j \neq i} ker(f_{j})$ such that $b^{f_{i}} = a$. Note that $H_{i}tb = H_{i}t$ for any $t \in T_{i}$, then
	$$b^{\varphi} = (e, \dots, e, a^{\varphi},\dots, a^{\varphi}, e,\dots, e) = a^{\varphi x_{i}}.$$
	The result follows.
\end{proof}

Now we prove

\begin{thmC}\label{thmB1}
	Let $A$ be a free abelian group.
	\begin{enumerate}
		\item  [(i)] Suppose $A$ has finite rank.
		\begin{enumerate}
			\item [(a)] If $A$ is a non-transitive self-similar group then $\Delta({A})$ is not finitely generated;
			\item[(b)] Let $A$  be a transitive self-similar subgroup of $\mathcal{A}_m$. Then the self-similar representation of $A$ extends to a non-transitive self-similar representation into $\mathcal{A}_{m+1}$ having orbit-type $(m,1)$ such that $\Delta(A)$, with respect to the second representation, contains a self-similar free abelian group of infinite enumerable rank.
		\end{enumerate}
		
		\item[(ii)] Suppose $A$ has infinite countable rank. Then $A$ can be realized as a self-similar group of orbit-type $(m,1),$ invariant with respect to $\Delta=\langle x_1,x_2 \rangle$.
	\end{enumerate}
\end{thmC}

\begin{proof}
	
	\begin{enumerate}
		\item [(i)-(a)] 
		
		Suppose that $K = \Delta(A)$ is finitely generated. Then $K = \Delta(K)$ and by Proposition \ref{sr} there exists a strongly recurrent $K$-data $({\bf m}, {\bf H}, {\bf F})$. Since $K$ is an abelian group, we have that $f_i|_{H_{i} \cap \bigcap_{j \neq i} ker(f_{j})}$ is injective, $ H_{i} \cap \bigcap_{j \neq i} ker(f_{j}) \simeq K$  and the index 
		$$\left[K:  H_{i} \cap \bigcap_{j \neq i} ker(f_{j})\right]$$
		is finite for all $i=1,\dots,s.$
		Therefore $ker(f_{1})\cap \cdots \cap ker(f_{s})\neq 1;$ a contradiction. \\

		\item[(i)-(b)] Let $f: H \rightarrow A$ be a simple virtual endomorphism, where $H$ is a subgroup of index $m$ in $A$. Consider the free abelian group of infinite rank $A^{\omega} = \oplus_{i = 1}^{\infty} A_{i}$ where $A_{i} = A$ for all $i $. Write $H_{1} = H \oplus ( \oplus_{i = 2}^{\infty} A_{i})$, $H_{2} = A^{\omega}$ and for $i=1,2$, define the homomorphisms  $f_{i}:H_{i}\rightarrow A^{\omega}$
		\begin{eqnarray*}
			f_{1} &:&\,(h, a_{2}, a_{3}, \dots) \mapsto (h^f, a_{2}, a_{3}, \dots);
		\end{eqnarray*}
		\begin{center}
		\end{center}
		\begin{eqnarray*}
			f_{2} &:&\,(a_{1}, a_{2}, a_{3}, \dots) \mapsto (a_{2}, a_{3}, a_{4}, \dots).
		\end{eqnarray*}
		\noindent It follows that the $A^{\omega}$-data 
		$$({\bf m + 1} = (m, 1), {\bf H} = (H_{1}, H_{2}), {\bf F} = \{f_{1}, f_{2}\})$$
		has a trivial ${\bf F}$-core. Let $\varphi: A^{\omega} \rightarrow \mathcal{A}_{m+1}$ be the self-similar representation of $A^{\omega}$ induced by the above $A^{\omega}$-data. Then,
		$$A_{1}^{\varphi} \leq (A^{\omega})^{\varphi} \leq \Delta(A_{1}^{\varphi})$$
		and $A_{1}^{\varphi}$ is self-similar. \\
		
		\item[(ii)]
		Let $m \geq 2$ be an integer and let $\{a_1,a_2,a_3,\dots\}$ be a free basis of $A$. Consider $H_1=\langle a_1^m,a_2,a_3,\dots\rangle$ and $H_2=A.$ Define the homomorphisms $f_i:H_i \to A$ ($i=1,2$) which extend the maps 
		$$
		f_{1} :\,a_{1}^{m}\mapsto a_{1}, \,\, a_{2i}\mapsto e\text{ }\left(i\geq 1\right), \,\, a_{2i-1}\mapsto a_i\text{ }\left(i\geq 2\right),$$
		$$f_{2} :\,a_{2i-1}\mapsto e\text{ }\left( i\geq 1\right) , \,\, a_{2i}\mapsto a_{i}\text{ }\left(i\geq 1\right).
		$$
		Note that the $A$-data
		${((m,1),(H_1,H_2),(f_1,f_2))}$ is strongly recurrent. By Proposition \ref{sr}, the induced representation is closed under $\Delta=\langle x_1,x_2\rangle$, where by definition
		$$a^{x_1}=(a, ..., a,e),\,\,\, a^{x_2}=(e, ..., e, a).$$
		Then
		$$A\simeq \langle \alpha_i \,\,| \,\,\, i=1,2,\dots\rangle,\,\,\,\, \text{where}$$
		$$\alpha_1=(e, ..., e, \alpha_{1},e)(1 \, 2 \, ... \, m),  \alpha_{2i-1}= {\alpha_i}^{x_1}\,\, \left( i\geq 2\right)\,,$$ $$\,\alpha_{2i}={\alpha_i}^{x_2}\,\,(i\geq 1).$$
	\end{enumerate}
\end{proof}

\section{Self-similar cyclic groups}

\subsection{A procedure for computing the centralizer of self-similar cyclic groups.}

Let $A$ be a cyclic self-similar group of $\mathcal{A}_m$ of orbit-type $(m_1,\dots,m_s)$. Then $A$ is generated by
$$a=(a^{i_{11}},\dots,a^{i_{1m_1}}, \dots, a^{i_{s1}},\dots,a^{i_{sm_s}})\sigma_{(1)}\cdots\sigma_{(s)},$$
$$=a_{[1]} \cdots a_{[s]}, \,\, \text{where}$$
$$a_{[1]}=(a^{i_{11}},\dots,a^{i_{1m_1}}, e, \dots, e)\sigma_{(1)}, \dots, a_{[s]}= (e,\dots,e,a^{i_{s1}},\dots,a^{i_{sm_s}})\sigma_{(s)}$$
for some integers $i_{kl}.$ 

Thus we have 
$$P_{(1)}=\langle \sigma_{(1)}\rangle  \,\,, \dots,\,\,P_{(s)}= \langle \sigma_{(s)} \rangle,$$
and $ A_{[i]}=\langle a_{[i]} \rangle$ for $1\leq i \leq s.$

By Proposition \ref{teo1},
$$C(A)=Stab_{C(A)} (1) B(A) R(A), \text{where}$$
$$B(A)=A_{[1]}\cdots A_{[s]} .$$

\noindent \textbf{Step 1.} {The form of ${Stab_C(1)}$.}

Let $c=(c_{11},\dots,c_{1m_1},\dots,c_{s1},\dots,c_{sm_s})$ be an element of $Stab_C(1)$. The relation $ca=ac$ translates to
$$c_{jk}a^{i_{jk}}=a^{i_{jk}}c_{j(k)^{\sigma_j}}, \,\, j=1,\dots, s\,\,,\,\, k=1,\dots,m_j.$$
From these relations we obtain  
$$Stab_{C(A)}(1)=\mathbf{C_1} \times \cdots \times \mathbf{C_{s}}, $$
where $\mathbf{C_{1}}, ..., \mathbf{C_{s}}$ are respectively the sets
$$\{(c_{11},(c_{11})^{a^{i_{11}}},(c_{11})^{a^{i_{11}+i_{12}}},\dots,(c_{11})^{a^{i_{11}+\cdots+i_{1,m_1-1}}}) \,\, | \,\, c_{11} \in C(a^{i_{11}+\cdots+i_{1m_1}})\},$$
$$\vdots $$
$$\{(c_{s1},(c_{s1})^{a^{i_{s1}}},(c_{s1})^{a^{i_{s1}+i_{s2}}},\dots,(c_{s1})^{a^{i_{s1}+\cdots+i_{s,m_s-1}}}) \,\, | \,\, c_{s1} \in C(a^{i_{s1}+\cdots+i_{sm_s}})\}.$$

\noindent \textbf{Step 2.} Computing $R(A)$. 

Let $r=(r_{11},\dots,r_{1m_1},\dots,r_{s1},\dots,r_{sm_s})\xi \in R(A)$, where $\xi \in S(A)$. Identify $$11:=\mathbf{1},12:=\mathbf{2},\dots,1m_1:=\mathbf{m_1},21:=\mathbf{m_1+1}, \dots, sm_s:=\mathbf{m}.$$
Then the relation $ra=ar$ is equivalent to the system of equations 
$$r_ja_{j^{\sigma}}=a_jr_{j^{\sigma}}, j=\mathbf{1},\dots,\mathbf{m},$$
which may or not have a solution.

It may be easier to work with the right normal form of a which is found by conjugation, as follows. The conjugator will be an infinite product whose first term is

$$g=(g_{(1)},g_{(1)},\dots,g_{(s)}) \in \mathcal{A}_m$$
defined by  
$$g_{(i)}=(a_{i1},e,{(a_{i2})}^{-1},{(a_{i2}a_{i3})}^{-1},\dots,{(a_{i2}\cdots a_{im_{i-1}})}^{-1}).$$

Then, $a^g=b$, where 

$$b=(b_{(1)},b_{(2)},\dots,b_{(s)})\sigma_1\sigma_2\cdots\sigma_s,$$
$$b_{(i)}=(e,e,\dots,e,c_i)$$ 
$$c_i=a_{i2}a_{i3}\cdots a_{i,m_{i-1}}a_{i1}.$$
The second term of the conjugator is 
$h={h_{11}}^{x_1}\cdots{h_{s1}}^{x_s},$ where
$h_{11},\dots,h_{s1}$ are computed to normalize $c_1,\dots,c_s$ as $g$ was. The process continues iterating these steps. 

In the case of a self-similar $A=\langle a\rangle$, we have
$$a_{(k)}=(a^{i_{k1}},\dots,a^{i_{km_k}}), \,\,\, k=1,\dots,s,$$ 
$$g_{(k)}=(a_{i1},e,{(a_{i2})}^{-1},{(a_{i2}a_{i3})}^{-1},\dots,{(a_{i2}\cdots a_{im_{i-1}})}^{-1}).$$
Define $j_k=i_{k1}+i_{k2}+\cdots+i_{km_k}$, $k=1,\dots,s$, then by Step 1 of the procedure,
$$Stab_{C(b)}(1)={C(a^{j_1})}^{x_1}\times {C(a^{j_2})}^{x_2} \times \cdots \times {C(a^{j_s})}^{x_s} $$
and so,
$$Stab_{C(a)} (1)= {\left(Stab_{C(b)} (1)\right)}^{g^{-1}}.$$

In the transitive case (that is, $s=1$), we have that
$C(a)={\left({C(a^{j_1})}^{x_1}\right)}^{g^{-1}}\langle a \rangle$.

\subsubsection{The $m$-adding machine of multiplicity $s$}

This machine is an element of $\mathcal{A}_{ms}$
defined by
$$a=(a_{(1)},\dots,a_{(s)})\sigma,$$
where $a_{(1)}=a_{(2)}=\cdots= a_{(s)}=(e,\dots,e,a)$
and
$$\sigma=(1 \,2 \cdots m)(m+1\, \,m+2 \, \cdots \, 2m) \cdots ((s-1)m+1\, \, (s-1)m+2 \, \cdots \, sm).$$
Then $A=\langle a\rangle$ is a self-similar cyclic group and $R(A)=S(A) \simeq Sym(s).$ 
Also
$$C=C_{\mathcal{A}_{ms}} (A)
= B(A) Stab_{C} (1) R(A),$$
$$Stab_C (1)=C^{x_1}\cdot C^{x_2}\cdots C^{x_s}.$$
By substitution and collection processes, as in the proof of Theorem A, we have
$$C=\prod_{w \in \Delta, \,\, |w|\leq k} {B(A)}^w
\prod_{w \in \Delta, \,\, |w|=k} {Stab_C (1)}^w
\prod_{ w \in \Delta, \,\, |w|\leq k} {R(A)}^w.$$
In the limit we obtain
$$C=\overline{\Delta({ B(A)})} \cdot \overline{\Delta(R(A))}=A^* \cdot \overline{ \Delta(R(A))}.$$
We note that $(a_{[i]})^m=a^{x_i}$ for $1\leq i \leq s$ and $A^*$ is torsion-free. Furthermore, $\overline{\Delta(R(A))} \simeq \mathcal{A}_s.$
Since $A^*$ is self-centralizing,
it is a faithful module for $\mathcal{A}_s$ and is torsion-free. 

As we see, $A^*$ is a new type of module for $\mathcal{A}_s$ which in our view deserves to be studied in greater detail. For natural modules of $\mathcal{A}_s$, see \cite{Sid}.

\subsection{Centralizer of self-similar cyclic subgroups of $\mathbf{Aut(\mathcal{T}_4)}$}

We describe the centralizer of intransitive cyclic self-similar subgroups of ${Aut(\mathcal{T}_4)}$. The analysis is based on the orbit-type of the group. 
We will use the following lemma from \cite{PVV}.

\begin{lemma}\label{conj}
	Let $a \in A \leqslant Aut\left( \mathcal{T}_{m}\right) $ and $\xi $ be a unit in the ring of
	integers $\mathbb{Z}_{n}$, where $n$ is the exponent of the group $P(A)$. Then $a^{\xi }$ is conjugate to $a$ by a computable $g\in Aut\left( \mathcal{T}_{m}\right) $. 
\end{lemma}

\begin{center}
	\textbf{Orbit-type (2,2)} 
\end{center}  The alphabet $Y=\left\{ 1,2,3,4\right\} $ is the union two $A$-orbits, say $%
O_{(1)}=\left\{ 1,2\right\} ,O_{(2)}=\left\{ 3,4\right\}.$ We may assume $P(A)=\left\langle \left( 1\,2\right) \left( 3\,4\right) \right\rangle
,P_{(1)}=\left\langle \left( 1\,2\right) \right\rangle ,P_{(2)}=\left\langle
\left( 3\,4\right) \right\rangle \text{.}$

There exist integers $i_{1},i_{2},i_{3},i_{4}$ such that the group $A$ is
generated by 
\begin{equation*}
	a=\left( a^{i_{1}},a^{i_{2}},a^{i_{3}},a^{i_{4}}\right) \left( 1\,2\right)
	\left( 3\,4\right) \text{.}
\end{equation*}%
Here%
\begin{equation*}
	C_{Sym\left( 4\right) }\left( P\right) =P_{(1)}P_{(2)}S\text{ where }%
	S=\left\langle \left( 1\,3\right) \left( 2\,4\right) \right\rangle \text{,}
\end{equation*}%
\begin{eqnarray*}
	A_{[1]} &=&\left\langle \left( a^{i_{1}},a^{i_{2}},e,e\right) \left(
	1\,2\right) \right\rangle ,A_{[2]}=\left\langle \left(
	e,e,a^{i_{3}},a^{i_{4}}\right) \left( 3\,4\right) \right\rangle , \end{eqnarray*}
\begin{eqnarray*}
	C\left( A\right) &=&\text{ }Stab_{C\left( A\right) }\left( 1\right)B(A) R(A) \text{, } 
\end{eqnarray*}

\begin{eqnarray*}
	B(A)=A_{[1]}A_{[2]},
\end{eqnarray*}
\begin{eqnarray*}
	R(A)= \langle r \rangle, \,\, r&=&\left( r_{1},r_{2},r_{3},r_{4}\right) \xi \text{ and }\xi \in \left\langle
	\left( 1\,3\right) \left( 2\,4\right) \right\rangle .
\end{eqnarray*}

\noindent\textbf{Computing $\mathbf{Stab_{C\left( A\right) }\left(
		1\right)}$ and $\mathbf{r}$.} Let $\ c=\left( c_{1},c_{2},c_{3},c_{4}\right) $. Then $c\in C\left(
A\right) $ if and only if 

\begin{eqnarray*}
	c_{2} &=&\left( c_{1}\right) ^{a^{i_{1}}},c_{1}\in C\left(
	a^{i_{1}+i_{2}}\right) , \\
	c_{4} &=&\left( c_{3}\right) ^{a^{i_{3}}},c_{3}\in C\left( a^{i_{3}+i_{4}}%
	\text{}\right).
\end{eqnarray*}%
Denote $j_{1}=i_{1}+i_{2},$ $j_{3}=i_{3}+i_{4}$. Therefore,

\begin{equation*}
	Stab_{C\left( A\right) }\left( 1\right) =\mathbf{C}_{1}\times \mathbf{C}_{3},
\end{equation*}
\begin{eqnarray*}
	\mathbf{C}_{1} &=&\left\{ \left( c_{1},\left( c_{1}\right)
	^{a^{i_{1}}}\right) \mid c_{1}\in C\left( a^{j_{1}}\right) \right\} , \\
	\mathbf{C}_{3} &=&\left\{ \left( c_{3},\left( c_{3}\right)
	^{a^{i_{3}}}\right) \mid c_{3}\in C\left( a^{j_{3}}\right) \right\} \text{.}
\end{eqnarray*}

Thus the problem of description of $Stab_{C\left( A\right) }\left( 1\right) $
translates to a description of $C\left( a^{j_{1}}\right) ,C\left(
a^{j_{3}}\right) $, a problem we will deal with later.

\noindent\textbf{The form of $\mathbf{r}$.} Suppose $r=\left( r_{1},r_{2},r_{3},r_{4}\right) \left( 1\,3\right) \left(
2\,4\right) \in C\left( A\right) $. Then $ra=ar$ translates to the equations 
\begin{eqnarray*}
	r_{2} &=&a^{-i_{1}}r_{1}a^{i_{3}},\text{ }\left( a^{j_{1}}\right)
	^{r_{1}}=a^{j_{3}} \\
	r_{4} &=&a^{-i_{3}}r_{3}a^{i_{1}},\text{ }\left( a^{j_{3}}\right)
	^{r_{3}}=a^{j_{1}}\text{.}
\end{eqnarray*}%
Combining the 2nd and the 4th equation we obtain $r_{1}r_{3}$ commutes with $
a^{j_{1}}$. We choose $r_{1}r_{3}=e$; that is, $r_{3}=r_{1}^{-1}$. With this
choice, we get 
\begin{equation*}
	r=\left( r_{1},\text{ }a^{-i_{1}}r_{1}a^{i_{3}},\text{ }r_{1}^{-1},\text{ }%
	a^{-i_{3}}r_{1}^{-1}a^{i_{1}}\right) \left( 1\,3\right) \left( 2\,4\right)
\end{equation*}%
where $\left( a^{j_{1}}\right) ^{r_{1}}=a^{j_{3}}$. Thus the existence of $r$
is reduced to whether or not $a^{j_{1}}$ is conjugate to $a^{j_{3}}$.

Factor $j_{1}=2^{k_{1}}l_{1},$ $j_{3}=2^{k_{3}}l_{3}$ where $k_{1},k_{3}\geq
0$ and $l_{1},l_{3}$ odd integers- we agree that $k_{i}$ is infinite if and
only if $j_{i}=0$ . Then $a^{j_{1}}$ is conjugate to $a^{j_{3}}$ if and only
if $k_{1}=k_{3}$ .

\noindent\textbf{Describing $\mathbf{C\left( a^{j_{1}}\right) ,C\left( a^{j_{3}}\right)}.$} By Lemma \ref{conj} there are $g_{1},g_{3}\in Aut\left( \mathcal{T}_{4}\right) $ such that $%
a^{l_{1}}=a^{g_{1}},a^{l_{3}}=a^{g_{3}}$. Let $g=g_{1}^{-1}g_{3}.$
We start with 
\begin{eqnarray*}
	a^{j_{1}} &=&\left( a^{l_{1}}\right) ^{2^{k_{1}}}=\left( a^{g_{1}}\right)
	^{2^{k_{1}}}=a^{2^{k_{1}}g_{1}}, \\
	a^{j_{3}} &=&\left( a^{l_{3}}\right) ^{2^{k_{3}}}=\left( a^{g_{3}}\right)
	^{2^{k_{3}}}=a^{2^{k_{3}}g_{3}}.
\end{eqnarray*}%
Then,%
\begin{eqnarray*}
	a^{2} &=&\left( a^{j_{1}},a^{j_{1}},a^{j_{3}},a^{j_{3}}\right) \text{,} \\
	&=&\left(
	a^{2^{k_{1}}g_{1}},a^{2^{k_{1}}g_{1}},a^{2^{k_{3}}g_{3}},a^{2^{k_{3}}g_{3}}%
	\right)
\end{eqnarray*}%
and
$$
C\left( a^{2}\right) =$$ $$\left( C\left( a^{2^{k_{1}}}\right) ^{g_{1}}\times
C\left( a^{2^{k_{1}}}\right) ^{g_{1}}\times C\left( a^{2^{k_{3}}}\right)
^{g_{3}}\times C\left( a^{2^{k_{3}}}\right) ^{g_{3}}\right) \left\langle
\left( 1,2\right) ,\left( 3,4\right) \right\rangle \left\langle w^{\eta
}\right\rangle,
$$
where $w=\left( g,g,g^{-1},g^{-1}\right) \left( 1\,3\right) \left( 2\,4\right) 
$ and $\eta =0$ if $k_{1}\not=k_{3}$, $\eta =1$ if $k_{1}=k_{3}$.

\bigskip There are basically three situations: $j_{1},j_{3}$ even ; $%
j_{1},j_{3}$ odd; $j_{1}$ odd , $j_{3}$ even.

(i) Let $j_{1},j_{3}$ be even. Then successive substitutions of even powers
of $a$ into the above expression shows $a^{2}\in \cap _{i\geq 1}Stab(i)$ and
therefore, $a^{2}=e$ and $j_{1}=j_{3}=0$, that is $i_{2}=-i_{1},i_{4}=-i_{3}$%
; so, 
\begin{equation*}
	a=\left( a^{i_{1}},a^{-i_{1}},a^{i_{3}},a^{-i_{3}}\right) \left( 1\,2\right)
	\left( 3\,4\right) \text{.}
\end{equation*}%
Thus,%
\begin{equation*}
	B(A)=\left\langle \left( a^{i_{1}},a^{-i_{1}},e,e\right) \left( 1\,2\right)
	,\left( e,e,a^{i_{3}},a^{-i_{3}}\right) \left( 3\,4\right) \right\rangle ,
\end{equation*}%
Also,%
\begin{equation*}
	Stab_{C\left( A\right) }\left( 1\right) =\mathbf{C}_{1}\times \mathbf{C}_{3},
\end{equation*}%
\begin{eqnarray*}
	\mathbf{C}_{1} &=&\left\{ \left( c_{1},\left( c_{1}\right)
	^{a^{i_{1}}}\right) \mid c_{1}\in Aut\left( \mathcal{T}_{4}\right) \right\} , \\
	\mathbf{C}_{3} &=&\left\{ \left( c_{3},\left( c_{3}\right)
	^{a^{i_{3}}}\right) \mid c_{3}\in Aut\left( \mathcal{T}_{4}\right) \right\} \text{,}
\end{eqnarray*}%
and we may choose
\begin{equation*}
	r=\left( e,\text{ }a^{-i_{1}+i_{3}},\text{ }e,\text{ }a^{i_{1}-i_{3}}\right)
	\left( 1\,3\right) \left( 2\,4\right) \text{.}
\end{equation*}

\bigskip (ii) Let $j_{1},j_{3}$ be odd. Then, 
\begin{equation*}
	Stab_{C\left( A\right) }\left( 1\right) =\mathbf{C}_{1}\times \mathbf{C}_{3},
\end{equation*}%
\begin{eqnarray*}
	\mathbf{C}_{1} &=&\left\{ \left( c_{1},\left( c_{1}\right)
	^{a^{i_{1}}}\right) \mid c_{1}\in C\left( A\right) ^{g_{1}}\right\} , \\
	\mathbf{C}_{3} &=&\left\{ \left( c_{3},\left( c_{3}\right)
	^{a^{i_{3}}}\right) \mid c_{3}\in C\left( A\right) ^{g_{3}}\right\},
\end{eqnarray*}%

\begin{equation*}
	B(A)=\left\langle \left( a^{i_{1}},a^{i_{2}},e,e\right) \left( 1\,2\right)
	,\left( e,e,a^{i_{3}},a^{i_{4}}\right) \left( 3\,4\right) \right\rangle ,
\end{equation*}
and
\begin{equation*}
	r=\left( r_{1},\text{ }a^{-i_{1}}r_{1}a^{i_{3}},\text{ }r_{1}^{-1},\text{ }%
	a^{-i_{3}}r_{1}^{-1}a^{i_{1}}\right) \left( 1\,3\right) \left( 2\,4\right)
\end{equation*}%
where $\left( a^{j_{1}}\right) ^{r_{1}}=a^{j_{3}}$. We may choose $g_1,g_3$ conjugators such that $a^{j_1}=a^{g_1},$ $a^{j_3}=a^{g_3}$ and $r_{1}=g={g_1}^{-1}g_3$
and so, 
\begin{equation*}
	r=\left( g,\text{ }a^{-i_{1}}ga^{i_{3}},\text{ }g^{-1},\text{ }%
	a^{-i_{3}}g^{-1}a^{i_{1}}\right) \left( 1\,3\right) \left( 2\,4\right) \text{.}
\end{equation*}

(iii) Let $j_{1}$ odd, $j_{3}$ even. Therefore, $%
a^{j_{1}}=a^{g_{1}},a^{j_{3}}=a^{2^{k_{3}}g_{3}}$ where $k_{3}\geq 1$ $.$

As $a^{j_{1}}$ is not conjugate to $a^{j_{3}}$, $r=e,$ and so, 
\begin{equation*}
	C\left( A\right) =\text{ }Stab_{C\left( A\right) }\left( 1\right) B(A),
\end{equation*}
\begin{eqnarray*}
	Stab_{C\left( A\right) }\left( 1\right) &=&\mathbf{C}_{1}\times \mathbf{C}%
	_{3}, \\
	\mathbf{C}_{1} &=&\left\{ \left( c_{1},\left( c_{1}\right)
	^{a^{i_{1}}}\right) \mid c_{1}\in C\left( A\right) ^{g_{1}}\right\} , \\
	\mathbf{C}_{3} &=&\left\{ \left( c_{3},\left( c_{3}\right)
	^{a^{i_{3}}}\right) \mid c_{3}\in C\left( a^{2^{k_{3}}}\right)
	^{g_{3}}\right\} \text{,}
\end{eqnarray*} 
\begin{equation*}
	B(A)=\left\langle \left( a^{i_{1}},a^{i_{2}},e,e\right) \left( 1\,2\right)
	,\left( e,e,a^{i_{3}},a^{i_{4}}\right) \left( 3\,4\right) \right\rangle .
\end{equation*}\\

\begin{center}
	\textbf{Orbit-type (2,1,1)}
\end{center}

\noindent In this case, we may assume
$$ A=\langle a=(a^{i_1},a^{i_2},a^{i_3},a^{i_4})(1\,2)(3)(4)\rangle,$$
where $i_1,i_2,i_3,i_4$ are integers.

\noindent Let  $P=P_{(1)}=\langle (1\,,2)\rangle$, $P_{(2)}=P_{(3)}=1$. Then $$C_{Sym(4)}(P)=P_{(1)}P_{(2)}P_{(3)}S=PS,$$
where $S=\langle (3\,4)\rangle$ and
$$A_{[1]}=\langle (a^{i_1},a^{i_2},e,e)(1\,2)\rangle\,\,, \,\, A_{[2]}= \langle (e,e,a^{i_3},e)\rangle,$$
$$A_{[3]} =\langle (e,e,e,a^{i_4})\rangle \,\,,\,\, B(A)=A_{[1]}A_{[2]}A_{[3]}.$$ 
It follows that $C(A)= Stab_C(1)B(A)\langle r \rangle,$
where $r=(r_1,r_2,r_3,r_4)\xi$ and $ \xi\in \langle(3\,4)\rangle$. \\

\noindent\textbf{Computing $\mathbf{Stab_{C\left( A\right) }\left(
		1\right) }$ and $\mathbf{r}$.}
If $\xi=1,$ then $r \in Stab_C(1)$ and $C(A)=Stab_C(1)B(A)$. Assume that $\xi=(3\,4)$; that is, $r=(r_1,r_2,r_3,r_4)(3\,4)$.  If there exists a solution $r$, we can choose it to be
$$r=(e,e,r_3,{r_3}^{-1})(3\,4). $$

Factor $i_3=2^{k_3}l_3$, $i_4=2^{k_4}l_4$, where $l_3$ and $l_4$ are odd integers. Then $a^{i_{3}}$ is conjugate to $a^{i_{4}}$ if and only if $k_{3}=k_{4}$.
Here
$$Stab_{C(A)}(1)=\mathbf{C_1} \times \mathbf{C_2} \times \mathbf{C_3},  $$ 
where
$$\mathbf{C_1}=\{(c_1,(c_1)^{a^{i_1}}) \,\, | \,\, c_1 \in C(a^{j_1})\},\,\,\,\mathbf{C_2}=C(a^{i_3}),\,\,\,\mathbf{C_3}=C(a^{i_4}).$$
\noindent\textbf{Description of $\mathbf{C(a^{j_1})}$, $\mathbf{C(a^{i_3})}$ e $\mathbf{C(a^{i_4})}$.} Note that $a^2=(a^{j_1},a^{j_1},a^{2i_3}, a^{2i_4})$. Write $j_1=2^{k_1}l_1$, where $l_1$ is a odd number. By
Lemma \ref{conj} there are conjugators $g_1,g_3,g_4 $ such that $a^{l_1}=a^{g_1}$, $a^{l_3}=a^{g_3}$ and $a^{l_4}=a^{g_4}$. Then

$$ a^{j_1}=a^{2k_1l_1}={(a^{l_1})}^{2^{k_1}}={(a^{2^{k_1}})}^{g_1}.$$
In the same way, $ a^{i_3}={(a^{2^{k_3}})}^{g_3}$ and $ a^{i_4}={(a^{2^{k_4}})}^{g_4}$.
Thus, 
$$C(a^2)=\left({C(a^{2^{k_1}})}^{g_1}, {C(a^{2^{k_1}})}^{g_1}, {C(a^{2^{k_3+1}})}^{g_3},{C(a^{2^{k_4+1}})}^{g_4}\right)\langle(1\,2)\rangle \langle d \rangle.$$

\begin{enumerate}
	\item [(i)] $\mathbf{a^{i_3}}$ \textbf{conjugate to}  $\mathbf{a^{i_4}.}$ 
	In this case, $k_3=k_4$. Define $g={g_3}^{-1}g_4$. Then 
	$${(a^{i_3})}^{g}={(a^{i_3})}^{{g_3}^{-1}g_4}={(a^{2^{k_3}})}^{g_4}={(a^{2^{k_4}})}^{g_4}={a^{i_4}}.$$
	So we can choose $r_3=g$ and $d=(1,1,g,{g}^{-1})(3\,,4).$
	Here we have that
	$$\mathbf{C_1}=\left\{(c_1,(c_1)^{a^{i_1}}) \,\, | \,\, c_1 \in {C(a^{2^{k_1}})}^{g_1}\right\},$$
	$$\mathbf{C_2}=  {C(a^{2^{k_3}})}^{g_3}\,\,\,\text{,}\,\,\, \mathbf{C_3}=  {C(a^{2^{k_4}})}^{g_4},$$
	$$B(A)= \langle  (a^{i_1},a^{i_2},e,e)(1\,2), (e,e,a^{i_3},e), (e,e,e,a^{i_4})\rangle.$$
	\item [(ii)]  $\mathbf{a^{i_3}}$ \textbf{not conjugate to}  $\mathbf{a^{i_4}.}$ In this case $r$ does not exist and $C(A)=Stab_{C(A)}(1)B(A),$ where
	$$\mathbf{C_1}=\left\{(c_1,(c_1)^{a^{i_1}}) \,\, | \,\, c_1 \in {C(a^{2^{k_1}})}^{g_1}\right\},$$
	$$\mathbf{C_2}=  {C(a^{2^{k_3}})}^{g_3}\,\,\,\text{,}\,\,\, \mathbf{C_3}=  {C(a^{2^{k_4}})}^{g_4},$$
	$$B(A)= \langle  (a^{i_1},a^{i_2},e,e)(1\,2), (e,e,a^{i_3},e), (e,e,e,a^{i_4})\rangle.$$
\end{enumerate}

\begin{center}
	\textbf{Orbit-type (3,1).}
\end{center}

Here we may assume the two orbits to be $O_{(1)}=\{1,2,3\}$ and $O_{(2)}=\{4\}$, and
$$ A=\langle a=(a^{i_1},a^{i_2},a^{i_3},a^{i_4})(1\,2\,3)(4)\rangle,$$
where $i_1,i_2,i_3,i_4$ are integers.
Let  $P=P_{(1)}=\langle (1\,2\,3)\rangle$, $P_{(2)}=1$. Then $C_{Sym(4)}(P)=P$ and
$$A_{[1]}=\langle (a^{i_1},a^{i_2},a^{i_3},e)(1\,,2\,,3)\rangle\,\,, \,\, A_{[2]}= \langle (e,e,e,a^{i_4})\rangle\,\,,\,\,B(A)=A_{[1]}A_{[2]};$$
it follows that $C(A)= Stab_{C(A)}(1)B(A)$. \\
\noindent \textbf{Computing $\mathbf{Stab_{C\left( A\right) }\left(
		1\right) }$.}
Let $c=(c_1,c_2,c_3,c_4) \in Stab_C(1)$. The relation $ac=ca$ produces the equations
$$c_1a^{i_1}=a^{i_1}c_2 \,\, , \,\, c_2a^{i_2}=a^{i_2}c_3$$ 
$$c_3a^{i_3}=a^{i_3}c_1 \,\, , \,\, c_4a^{i_4}=a^{i_4}c_4;$$ 
which are equivalent to
$$ c_2={c_1}^{a^{i_1}} \,\, , \,\, c_3={c_1}^{a^{i_1+i_2}}$$
$$c_1 \in C(a^{j_3})\,\,  , \,\,  c_4 \in C(a^{i_4}),$$ 
where $j_3=i_1+i_2+i_3.$

\noindent So, 
$$Stab_C(1)=\left\{(c_1,(c_1)^{a^{i_1}},(c_1)^{a^{i_1+i_2}},c_4) \,\, | \,\, c_1 \in  C(a^{j_3}), c_4 \in C(a^{i_4})\right\}. $$ 
Then, $Stab_C(1)=\mathbf{C_1} \times \mathbf{C_2}$, where
$$\mathbf{C_1}=\{(c_1,(c_1)^{a^{i_1}},(c_1)^{a^{i_1+i_2}}) \,\, | \,\, c_1 \in C(a^{j_3})\}\,\,\,,\,\,\,\mathbf{C_2}=C(a^{i_4}).$$
Factor $j_3=3^{k}l$, with $3\nmid l$.
\begin{enumerate}

	\item [(i)] $\mathbf{k=0.}$ In this case, $j_3=l$ and by Lemma \ref{conj} there is $g \in \mathcal{A}_4$ such that $a^{l}=a^{g}$; so, $C(a^{j_3})={C(a)}^{g}$.
	
	\item [(ii)] $\mathbf{k\neq 0 }.$ By Lemma \ref{conj} there is a conjugator $g$ such that $a^{l}=a^{g}$ and $C(a^{j_3})={C(a^{3^k})}^{g}$.
	It remains to describe $C(a^3)$. Observe that $$a^3=(a^{j_3},a^{j_3},a^{j_3},a^{3i_4}).$$
	Thus, 
	\begin{equation*}
		C\left( a^{3}\right) =\left( C\left( a^{3^{k_{}}}\right) \times
		C\left( a^{3^{k_{}}}\right)\times C\left( a^{3^{k_{}}}\right)
		\times C\left( a^{3i_4}\right) ^{_{}}\right) \left\langle
		\left( 1\,2\,3\right)\right\rangle.
	\end{equation*}%
\end{enumerate}

Now we repeat our previous calculations.

\end{document}